\documentclass[reqno]{amsart}
\usepackage{amsmath, amssymb, amsthm, epsfig}
\usepackage{hyperref, latexsym}
\usepackage{url}
\usepackage[mathscr]{euscript}
\usepackage{mathabx}
\usepackage{float}
\usepackage{color}
\usepackage{fullpage} 
\usepackage{setspace}

\def\today{\ifcase\month\or
	January\or February\or March\or April\or May\or June\or
	July\or August\or September\or October\or November\or December\fi
	\space\number\day, \number\year}

\newtheorem{theorem}{Theorem}

\newtheorem{lemma}[theorem]{Lemma}

\theoremstyle{definition}

\theoremstyle{remark}

\newcommand{\mc}{\mathcal}
\newcommand{\A}{\mc{A}}

\newcommand{\hh}{\tfrac12}

\newcommand{\dt}{\text{\rm d}t}
\renewcommand{\d}{\text{\rm d}}

\newcommand{\re}{{\rm Re}\,}

\begin{document}
	\title[Bounds for zeta on the 1-line   \\ under partial Riemann hypothesis]{Bounding zeta on the 1-line   \\ under the partial Riemann hypothesis}
\author{Andr\'es Chirre}


\address{Departamento de Ciencias - Sección Matemáticas, Pontificia Universidad Católica del Perú, Lima, Perú}
\email{cchirre@pucp.edu.pe}

	\allowdisplaybreaks
	\numberwithin{equation}{section}
	
	\maketitle  
	
	\begin{abstract}
	We provide explicit bounds in the theory of the Riemann zeta-function at the line $\re{s}=1$, assuming that the Riemann hypothesis holds until the height $T$. In particular, we improve some bounds, in finite regions, for the logarithmic derivative and the reciprocal of the Riemann zeta-function.
	\end{abstract}
	

	\smallskip

\section{Introduction}
A classical problem in analytic number theory is to find explicit bounds for the Riemann zeta-function. In particular, bounds at the line $\re{s}=1$ are of great interest, due to their usefulness in estimations to the M\"obius function and the von Mangoldt function. The main purpose of this paper is to obtain new bounds for the Riemann zeta-function in ranges where currently it is challenging to get computational verification. 

\subsection{Background} Let $\zeta(s)$ be the Riemann zeta-function. Unconditionally, it is known
that, as $t\to\infty$,
\begin{align} \label{11_24pm} 
 \dfrac{1}{\zeta(1+it)}=O(\log t),\,\,\,\,\,\,\,\, \mbox{and} \,\,\,\,\,\,\,\,\,\,\dfrac{\zeta'(1+it)}{\zeta(1+it)}=O(\log t).
\end{align} 
Currently, the best explicit bounds for  $1/\zeta(1+it)$ and $\zeta'(1+it)/\zeta(1+it)$ are given by
\begin{align}  \label{7-46pm}
	\left|\dfrac{1}{\zeta(1+it)}\right|\leq 42.9\log t, \,\,\,\,\,\,\,\, \mbox{and}\,\,\,\,\,\,\,\,\,\,	\left|\dfrac{\zeta'(1+it)}{\zeta(1+it)}\right|\leq 40.14\log t,
\end{align}
for $t\geq 133$. The first bound in \eqref{7-46pm} was established by Carneiro, Chirre, Helfgott, and Mejía-Cordero in \cite[Proposition A.2]{CCHM}, and the second bound was established by Trudgian in \cite{Trud}. There are improvements in the orders of magnitude of the mentioned estimates (see for instance \cite[p. 135]{Tit}), but it appears that those bounds are better when $t$ is astronomically large, and then they will not be useful for computational purposes.

On the other hand, assuming the Riemann hypothesis (RH), Littlewood proved in \cite{L2} that, as $t\to\infty$,
$$
\left|\dfrac{1}{\zeta(1+it)}\right|\leq\left(\dfrac{12e^\gamma}{\pi^2}+o(1)\right) \log\log t,
$$
where $\gamma$ is the Euler–Mascheroni constant. An explicit version of this result has been given by Lamzouri, Li, and
Soundararajan in \cite[p. 2394]{Sound}, establishing for $t\geq 10^{10}$ that
\begin{align*}  
	\left|\dfrac{1}{\zeta(1+it)}\right|\leq\dfrac{12e^\gamma}{\pi^2}\left(\log\log t - \log 2 + \dfrac{1}{2} + \dfrac{1}{\log \log t}+\dfrac{14\log\log t}{\log t}\right).
\end{align*}
Moreover, recently Chirre, Simoni\v{c}, and Val{\aa}s in \cite[Theorem 5]{ChirreSV}, under RH, proved for $t\geq10^{30}$ that
\begin{align*} 
	\left|\dfrac{\zeta'(1+it)}{\zeta(1+it)}\right|\leq 2\log\log t -0.4989 + \,5.35\,\dfrac{(\log\log t)^2}{\log t}. 
\end{align*}
Some generalizations of these estimates for families of $L$-functions can be found in \cite{Lumley,SimonicPalo}.

\subsection{Bounds for zeta under partial RH} In this paper, we are interested in obtaining bounds for the Riemann zeta-function, but only assuming a partial verification of RH. For $T>0$ we say that the Riemann hypothesis is true up to height $T$ (RH up to height $T$) if all non-trivial zeros $\rho=\beta+i\gamma$ of $\zeta(s)$ such that $|\gamma|\leq T$ satisfy $\beta=1/2$. The best current result is given by Platt and Trudgian in \cite{Platt}, who verified numerically in a rigorous way using interval arithmetic, that the Riemann hypothesis is true up to height $T=3\cdot 10^{12}$.  

\begin{theorem}\label{maintheorem}
	For a fixed $0<\delta<1$ define
	\begin{align}  \label{8_06pm}
		E_\delta(T)=\left(\dfrac{1}{\delta^2}+1\right) \dfrac{\log T}{2\pi T}.
	\end{align}
	Assume RH up to height $T\geq 10^9$. Then, for $10^6\leq t\leq  (1-\delta)T$ we have
\begin{align} \label{10_09pm}
\left|\dfrac{\zeta'(1+it)}{\zeta(1+it)}\right|\leq 2\log\log t+1.219 +\dfrac{16.108}{(\log\log t)^2}+1.057E_\delta(T).
\end{align}
	and
	\begin{align} \label{10_10pm}
		\left|\dfrac{1}{\zeta(1+it)}\right|\leq {2e^\gamma}\left(\log\log t +3.404 + \dfrac{9.378}{\log \log t}\right)\cdot\exp(0.793E_\delta(T)).
	\end{align} 
\end{theorem}

From Theorem \ref{maintheorem} we can also derive explicit versions of \eqref{11_24pm} in a finite but large range where computational verification is difficult. In fact, by Platt and Trudgian's result, we can take $T=3\cdot 10^{12}$, and letting $\delta=10^{-5}$ it follows, unconditionally, that for $10^6\leq t\leq 2.99997\cdot 10^{12}$ we have
$$
\left|\dfrac{\zeta'(1+it)}{\zeta(1+it)}\right|\leq 0.639\cdot \log t,\,\,\,\,\,\,\,\,\,\mbox{and} \,\,\,\,\,\,\,\,\,\,\,\,\dfrac{1}{|\zeta(1+it)|}\leq 2.506\cdot\log t.
$$  
This improves the results in \eqref{7-46pm} in the range $10^6\leq t\leq 2.99997\cdot 10^{12}$. 

\vspace{0.2cm}

We mention that \eqref{10_10pm} is derived from an upper bound for $|\log\zeta(1+it)|$ (see \eqref{8_24pm}), which also allows us to deduce that
\begin{align} \label{4_39pm}
	|\zeta(1+it)|\leq {2e^\gamma}\left(\log\log t +3.404 + \dfrac{9.378}{\log \log t}\right)\cdot\exp(0.793E_\delta(T)).
\end{align} 
However, currently the best unconditional explicit bound for $\zeta(1+it)$ is given by Patel, who proved in \cite[Theorem 1.1]{Patel} that for $t\geq 3$
\begin{align} \label{7_45pm}
	|\zeta(1+it)|\leq \min\left\{\log t, \dfrac{\log t}{2}+1.93, \dfrac{\log t}{5}+44.02\right\}\!.
\end{align} 
So \eqref{4_39pm} improves \eqref{7_45pm}, if we have the verification of RH up to height $T$ sufficiently large.

\vspace{0.5cm}

The proof of Theorem \ref{maintheorem} is carried out in Section \ref{12_19pm}, and it partially follows the conditional proofs of \cite[Section 13.2]{MV}. Here, an explicit formula is used, which is a formula that relates the zeros of $\zeta(s)$ and the prime numbers. This formula is unconditional and contains a certain sum involving the non-trivial zeros. Assuming RH, this sum is bounded without much effort. In our case, the novelty here is the way to bound the contribution of the non-trivial zeros, since we only assume RH up to height $T$. We will split this sum into two parts, the zeros with ordinates $|\gamma|\leq T$ and $|\gamma|>T$, and we will analyze them separately. These sums are studied in Section \ref{7_02pm}. We highlight that the proof of Theorem \ref{maintheorem} is short, and the constants involved can be improved slightly. 

Throughout the paper we will use the notation $\alpha=O^*(\beta)$, which means that $|\alpha|\leq \beta$.

\smallskip

\section*{Acknowledgements}
I am grateful to Harald Helfgott for encouraging me in this project and for helpful discussions related to the material in this paper.

\section{The sum over the non-trivial zeros} \label{7_02pm}

To bound the sum related to the non-trivial zeros of $\zeta(s)$ with ordinates $|\gamma|\leq T$, we will use the following lemma. 

\begin{lemma} \label{1_28pm} Assume RH up to height $T>0$. Then, for $t\geq 10^6$ and $1\leq\alpha\leq 3/2$ we have
\begin{align*} 			\sum_{|\gamma|\leq T}\dfrac{\alpha-\hh}{(\alpha-\hh)^2+(t-\gamma)^2} & \leq \re\dfrac{\zeta'(\alpha+it)}{\zeta(\alpha+it)}+\dfrac{\log t}{2}.
\end{align*}
\end{lemma}
\begin{proof}
Letting $s=\alpha+it$, and using the fractional decomposition of $\zeta(s)$ (see \cite[Corollary 10.14]{MV}) we get
\begin{align*} 
	\sum_{\gamma}\dfrac{\alpha-\re{\rho}}{(\alpha-\re{\rho})^2+(t-\gamma)^2} & = \re\dfrac{\zeta'(s)}{\zeta(s)}+\dfrac{1}{2}\,\re\dfrac{\Gamma'}{\Gamma}\bigg(\dfrac{s}{2}+1\bigg)-\dfrac{\log \pi}{2} + \dfrac{\alpha-1}{(\alpha-1)^2+t^2}.
\end{align*}
From the bound
\begin{align*} 
	\frac{\Gamma'}{\Gamma}(z)= \log z -\frac{1}{2z}+ O^*\left(\dfrac{1}{4|z|^2}\right)\!, \,\,\,\,\,\mbox{for}\,\,\,\, \re{z}\geq 0
\end{align*} (see \cite[Lemma 3.11, p. 67]{HaraldG}), it follows that
\begin{align}\label{12_30pm}
	\sum_{\gamma}\dfrac{\alpha-\re{\rho}}{(\alpha-\re{\rho})^2+(t-\gamma)^2}  \leq \re\dfrac{\zeta'(s)}{\zeta(s)}+\dfrac{\log t}{2}.
\end{align}
On the other hand, one can split the sum over the zeros, and using the fact that $\re{\rho}<1$ we get
\begin{align} \label{12_29pm}
	\begin{split} 
		\sum_{\gamma}\dfrac{\alpha-\re{\rho}}{(\alpha-\re{\rho})^2+(t-\gamma)^2} & = 	\sum_{|\gamma|\leq T}\dfrac{\alpha-\hh}{(\alpha-\frac{1}{2})^2+(t-\gamma)^2}+	\sum_{|\gamma|>T}\dfrac{\alpha-\re{\rho}}{(\alpha-\re{\rho})^2+(t-\gamma)^2}\\
		& \geq \sum_{|\gamma|\leq T}\dfrac{\alpha-\hh}{(\alpha-\frac{1}{2})^2+(t-\gamma)^2}.
	\end{split}
\end{align}
Combining \eqref{12_30pm} and \eqref{12_29pm} we arrive at the desired result.	
\end{proof}

To bound the sum related to the non-trivial zeros of $\zeta(s)$ with ordinates $|\gamma|>T$, we will use the auxiliary function 
\begin{align*}
	E(t,T):=\sum_{|\gamma|>T}\dfrac{1}{(\gamma-t)^2},
\end{align*} 
where $t$ does not coincide with an ordinate of a zero of $\zeta(s)$. This function measures, in a certain sense, the difference between the bounds under RH up to height $T$ and the bounds under RH. In fact, for a fixed $t\geq 0$ we see that
$$
\lim_{T\to\infty} E(t,T)=0.
$$
To estimate $E(t,T)$, the parameter $t$ must not be close to the ordinates of the zeros, and we need to take $T$ sufficiently large to reduce the contribution. Here, we will bound this term using a sum studied by Brent, Platt, and Trudgian in \cite{BPT}.

\begin{lemma}\label{5_05pm} Fix $0<\delta<1$ and $T\geq 10^9$. Then, for $0\leq t\leq (1-\delta)T$ we have
$$
	0<E(t,T)\leq E_\delta(T),
$$
where $E_\delta(T)$ was defined in \eqref{8_06pm}.
\end{lemma}
\begin{proof} Since $t\leq (1-\delta)T$, we find that
$$
E(t,T)=\sum_{|\gamma|>T}\dfrac{1}{(\gamma-t)^2}= \sum_{\gamma>T}\dfrac{1}{(\gamma-t)^2}+\sum_{\gamma>T}\dfrac{1}{(\gamma+t)^2}\leq\left(\dfrac{1}{\delta^2}+1\right)\sum_{\gamma>T}\dfrac{1}{\gamma^2}.
$$		
By \cite[Theorem 1 and Example 1]{BPT}, 
\begin{align*}
	\left| \sum_{\gamma\geq T}'\dfrac{1}{\gamma^2}-\dfrac{1}{2\pi}\int_{T}^\infty\dfrac{\log(t/2\pi)}{t^2}\dt\right|\leq \dfrac{0.14 + 0.56\log T}{T^2},
\end{align*}
where the prime symbol $'$ indicates that if $\gamma=T$, then it is counted with weight $1/2$. Thus
$$
\sum_{\gamma>T}\dfrac{1}{\gamma^2}\leq \dfrac{1}{2\pi T}\log\left(\dfrac{T}{2\pi}\right)+\dfrac{1}{2\pi T}+\dfrac{0.14 + 0.56\log T}{T^2}.
$$
Hence, using that $T\geq 10^9$ we conclude.
\end{proof}

\section{Proof of Theorem \ref{maintheorem}} \label{12_19pm}

\subsection{Bounding $\zeta'(s)/\zeta(s)$} Assume that RH is true up to height $T\geq 10^9$. Let $t\geq 10^6$ and $1\leq\alpha\leq 3/2$. Given $x, y \geq 2$ and $s=\alpha+it$, the unconditional formula \cite[Eq. 13.35]{MV} states that 
\begin{align} \label{identity}
	\dfrac{\zeta'(s)}{\zeta(s)}=-\sum_{\rho}\dfrac{(xy)^{\rho-s}-x^{\rho-s}}{(\rho-s)^2\log y}-\sum_{k=1}^\infty\dfrac{(xy)^{-2k-s}-x^{-2k-s}}{(2k+s)^2\log y}+\dfrac{(xy)^{1-s}-x^{1-s}}{(1-s)^2\log y}-\sum_{n\leq xy}\dfrac{\Lambda(n)}{n^s}w(n),
\end{align} 
where $w(n)$ is a function defined in  \cite[p. 433]{MV} satisfying that $|w(n)|\leq 1$. 
Let us bound each term on the right-hand side of \eqref{identity}. Since  $|(xy)^{\rho-s}-x^{\rho-s}|\leq x^{\re{\rho}-\alpha}(y^{\re{\rho}-\alpha}+1)$ and $\re{\rho}<1$ we have 
\begin{align*}
		\left|\sum_{\rho}\dfrac{(xy)^{\rho-s}-x^{\rho-s}}{(\rho-s)^2\log y}\right| &  =\left|\sum_{|\gamma|\leq T}\dfrac{(xy)^{\frac{1}{2}+i\gamma-s}-x^{\frac{1}{2}+i\gamma-s}}{(\frac{1}{2}+i\gamma-s)^2\log y}  +  \sum_{|\gamma|>T}\dfrac{(xy)^{\rho-s}-x^{\rho-s}}{(\rho-s)^2\log y}\right| \\& \leq \dfrac{x^{\frac{1}{2}-\alpha}(y^{\frac{1}{2}-\alpha}+1)}{\log y}\sum_{|\gamma|\leq T}\dfrac{1}{(\alpha-\frac{1}{2})^2+(t-\gamma)^2}  + \dfrac{x^{1-\alpha}(y^{1-\alpha}+1)}{\log y}\,E(t,T).
\end{align*}
Assuming that $10^6\leq t\leq (1-\delta)T$, Lemmas \ref{1_28pm} and \ref{5_05pm} lead us to 
\begin{align*}
	\left|\sum_{\rho}\dfrac{(xy)^{\rho-s}-x^{\rho-s}}{(\rho-s)^2\log y}\right|
	&\leq
	\dfrac{x^{\frac{1}{2}-\alpha}(y^{\frac{1}{2}-\alpha}+1)}{(\alpha-\frac{1}{2})\log y}\cdot\re\dfrac{\zeta'(s)}{\zeta(s)}+\dfrac{x^{\frac{1}{2}-\alpha}(y^{\frac{1}{2}-\alpha}+1)\log t}{2(\alpha-\frac{1}{2})\log y}+ \dfrac{x^{1-\alpha}(y^{1-\alpha}+1)}{\log y}\,E_\delta(T).
\end{align*}
We estimate the next terms in \eqref{identity} trivially as follows \begin{align*}
	\Bigg|\sum_{k=1}^\infty\dfrac{(xy)^{-2k-s}-x^{-2k-s}}{(2k+s)^2\log y} \Bigg|\leq \dfrac{0.3}{t^2}, \,\,\,\,\,\,\,\,\,\,\,\,\,\,\,\,\,\,\,\,\,
	\Bigg|\dfrac{(xy)^{1-s}-x^{1-s}}{(1-s)^2\log y}\Bigg| \leq \dfrac{2.9}{t^2},
\end{align*} 
and 
\begin{align*} 
	\left|\sum_{n\leq xy}\dfrac{\Lambda(n)}{n^s}w(n)\right|\leq \sum_{n\leq xy}\dfrac{\Lambda(n)}{n^\alpha}.
\end{align*}
Inserting these bounds in \eqref{identity} we arrive at
\begin{align*} 
	\left|\dfrac{\zeta'(s)}{\zeta(s)}\right|\leq & 
	\dfrac{x^{\frac{1}{2}-\alpha}(y^{\frac{1}{2}-\alpha}+1)}{(\alpha-\frac{1}{2})\log y}{\bigg|\dfrac{\zeta'(s)}{\zeta(s)}\bigg|}
	+\dfrac{x^{\frac{1}{2}-\alpha}(y^{\frac{1}{2}-\alpha}+1)\log t}{2(\alpha-\frac{1}{2})\log y} +\sum_{n\leq xy}\dfrac{\Lambda(n)}{n^\alpha}+\dfrac{x^{1-\alpha}(y^{1-\alpha}+1)}{\log y}\,E_\delta(T)+\dfrac{3.2}{t^2}.
\end{align*}

\smallskip

Now, let $\lambda_0=1.2784\ldots$ be the point where the function $\lambda\mapsto (1+e^\lambda)/\lambda$ 
reaches its minimum value $\A_0=3.5911\ldots$ in $(0,\infty)$. Take
\begin{align*} 
	y = \exp\bigg(\dfrac{\lambda_0}{\alpha-\frac{1}{2}}\bigg)\geq 2 \,\,\,\,\, \mbox{and} \,\,\,\,\, x= \dfrac{\log^2t}{y}\geq 2.
\end{align*}
Note that \begin{align*} 
		\dfrac{x^{\frac{1}{2}-\alpha}(y^{\frac{1}{2}-\alpha}+1)}{(\alpha-\frac{1}{2})\log y}=
	\A_0(\log t)^{1-2\alpha}<1,
\end{align*}	
and 
$$
\dfrac{x^{1-\alpha}(y^{1-\alpha}+1)}{\log y}\leq \dfrac{2}{\log y}=\dfrac{2\alpha-1}{\lambda_0}.
$$
Therefore, ordering conveniently we derive that
\begin{align} \label{12_45am}
	\begin{split}  
		\bigg|\dfrac{\zeta'(\alpha+it)}{\zeta(\alpha+it)}\bigg|&\leq\left(1+ \epsilon(\alpha,t)\right)\Bigg[\dfrac{\A_0}{2}{(\log t)^{2-2\alpha}} +\sum_{n\leq \log^2t}\dfrac{\Lambda(n)}{n^\alpha}+ \dfrac{(2\alpha-1)}{\lambda_0}E_\delta(T)+\dfrac{3.2}{t^2}\Bigg],
	\end{split}
\end{align}
where $\epsilon(\alpha,t)$ is defined as
$$
\epsilon(\alpha,t):=\dfrac{1}{\A_0^{-1}(\log t)^{2\alpha-1}-1}.
$$

\subsection{Bounding $\zeta'(1+it)/\zeta(1+it)$} Letting $\alpha=1$ in \eqref{12_45am}, it follows that
\begin{align*}  
		\bigg|\dfrac{\zeta'(1+it)}{\zeta(1+it)}\bigg|&\leq\left(1+ \epsilon(1,t)\right)\Bigg[\dfrac{\A_0}{2} +\sum_{n\leq \log^2t}\dfrac{\Lambda(n)}{n}+ \dfrac{E_\delta(T)}{\lambda_0}+\dfrac{3.2}{t^2}\Bigg].
\end{align*}
To bound the sum over the primes in the above expression, we use the estimate (see \cite[Lemma 10]{R})
\begin{align*}  
	\sum_{n\leq X}\dfrac{\Lambda(n)}{n}\leq \log X - \gamma +  \dfrac{1.3}{\log^2X}, \,\,\,\,\, \mbox{for} \,\,\,\,X>1.
\end{align*} 
Finally, ordering conveniently and using that $t\geq 10^6$ we arrive at \eqref{10_09pm}.

\subsection{Bounding $\log\zeta(1+it)$} By the fundamental calculus theorem, one has 
\begin{align*} 
	\log\zeta(1+it) = \log\zeta(\tfrac{3}{2}+it) -\int_{1}^{\frac{3}{2}}\dfrac{\zeta'(\alpha+it)}{\zeta(\alpha+it)}\,\d\alpha.
\end{align*} 
Using that $|\log\zeta(\tfrac{3}{2}+it)|\leq \log\zeta(\tfrac{3}{2})$ we obtain that
\begin{align*}   
	|\log\zeta(1+it)|\leq  \int_{1}^{\frac{3}{2}}\bigg|\dfrac{\zeta'(\alpha+it)}{\zeta(\alpha+it)}\bigg|\,\d\alpha +\log\zeta(\tfrac{3}{2}).
\end{align*}
To bound the right-hand side of the above expression we use the inequality $\epsilon(\alpha,t)\leq \epsilon(1,t)$ in $1\leq\alpha\leq 3/2$, and integrating \eqref{12_45am} from $1$ to $3/2$ we obtain that
\begin{align*} 
		|\log\zeta(1+it)|  \leq &\left(1+ \epsilon(1,t)\right)\Bigg[\dfrac{\A_0}{4\log\log t}+\sum_{n\leq \log^2t}\dfrac{\Lambda(n)}{n\log n}-\sum_{n\leq  \log^2t}\frac{\Lambda(n)}{n^{\frac{3}{2}}\log n} +\dfrac{3E_\delta(T)}{4\lambda_0}\Bigg]+\log\zeta(\tfrac{3}{2}),
\end{align*}
where we have used that $-\mathcal{A}_0/(4\log t\log\log t)+1.6/t^2<0$. Furthermore, using that
$$
\log\zeta(\tfrac{3}{2})= \sum_{n\geq 1}\dfrac{\Lambda(n)}{n^{\frac{3}{2}}\log n}=0.960\ldots, 
$$ 
we have for $t\geq 10^6$ that
$$
\left(1+ \epsilon(1,t)\right)\sum_{n> \log^2t}\frac{\Lambda(n)}{n^{\frac{3}{2}}\log n}\leq  \epsilon(1,t)\sum_{n\geq 1}\frac{\Lambda(n)}{n^{\frac{3}{2}}\log n}.
$$
This implies that
\begin{align} \label{10_29pm} 
	|\log\zeta(1+it)|  \leq &\left(1+ \epsilon(1,t)\right)\Bigg[\dfrac{\A_0}{4\log\log t}+\sum_{n\leq \log^2t}\dfrac{\Lambda(n)}{n\log n} +\dfrac{3E_\delta(T)}{4\lambda_0}\Bigg].
\end{align}
To bound the sum over the primes, we use
\cite[Eq. (3.30)]{Ro} to see that for $x>1$
\begin{align*}  
	\sum_{n\leq x}\dfrac{\Lambda(n)}{n\log n}\leq \sum_{p\leq x}\sum_{k= 1}^{\infty}\dfrac{1}{kp^k}=\log\prod_{p\leq x}\bigg(1-\dfrac{1}{p}\bigg)^{-1}\leq \log\log x+\gamma +  \dfrac{1}{\log^2x}.
\end{align*} 
Inserting this in \eqref{10_29pm} we arrive at
\begin{align*} 
	|\log\zeta(1+it)|  \leq &\left(1+ \epsilon(1,t)\right)\Bigg[\log\log\log t + \log(2e^\gamma)+\dfrac{\A_0}{4\log\log t}+\dfrac{1}{4(\log\log t)^2} +\dfrac{3E_\delta(T)}{4\lambda_0}\Bigg].
\end{align*}
Thus
\begin{align} \label{8_24pm}
|\log\zeta(1+it)|\leq \log\log\log t + \log(2e^\gamma)+\dfrac{3.404}{\log\log t}+0.793E_\delta(T).
\end{align}
Taking exponential in the above expression and using the inequality $e^x\leq 1+x+0.8093x^2$ for $0\leq x\leq 1.297$, we obtain that
\begin{align*} 
	\exp(|\log\zeta(1+it)|)\leq {2e^\gamma}\left(\log\log t +3.404 + \dfrac{9.378}{\log \log t}\right)\cdot\exp(0.793E_\delta(T)).
\end{align*} 
Since $\log|\zeta(1+it)|^{-1}\leq |\log\zeta(1+it)|$ and $\log|\zeta(1+it)|\leq |\log\zeta(1+it)|$, we deduce \eqref{10_10pm} and \eqref{4_39pm} respectively.


\end{document}